\documentclass[12pt]{article}
\usepackage{fancyhdr}
\usepackage[latin1]{inputenc}
\usepackage{amsfonts,amsthm}
\usepackage[colorlinks]{hyperref}
\usepackage{mathrsfs,amssymb,amsmath,wasysym}
\usepackage{enumerate,textcomp}
\usepackage{amsthm}
\usepackage{pifont}
\usepackage[numbers,sort&compress]{natbib}
\usepackage{setspace}
\usepackage{xcolor}
\usepackage[english]{babel}
\newtheorem{thm}{Theorem}
\newtheorem{lema}{Lemma}
\newtheorem{remark}{Remark}
\newcommand{\TheAuthor}{}
\newcommand{\Author}[1]{\renewcommand{\TheAuthor}{#1}}

\newcommand{\TheTitle}{}
\newcommand{\Title}[1]{\renewcommand{\TheTitle}{#1}}

\Author{Gomathi C S, Mohanapriya N, Vernold Vivin J}
\Title{On $r$-dynamic coloring on lexicographic product of star graphs}
\begin{document}

	%\blfootnote{This work is licensed under the 
	%\href{http://creativecommons.org/licenses/by-nd/4.0/}
	%{Creative Commons Attribution-NoDerivatives 4.0 International License}}   	

\parindent=8mm
\noindent

%\underline{{\noindent \bf \scriptsize Scientific Annals of Computer Science vol. xx, yyyy, pp. \thepage--\hyperlink{lastpage}{\pageref{LastPage}}}}
\vskip -3mm
\noindent
\vskip -1mm
\noindent

\vspace{1cm}
\begin{center}
{\Large\bf On $r$-dynamic coloring on lexicographic product of star graphs}
\end{center}
\vspace{4mm}

\begin{center}
{\large GOMATHI C S}\footnote{PG and Department of Mathematics, Kongunadu Arts and Science College, Coimbatore-641 029, Tamil Nadu, India, Email: {\tt gomathi9319@gmail.com}}
{\large MOHANAPRIYA N}\footnote{PG and Department of Mathematics, Kongunadu Arts and Science College, Coimbatore-641 029, Tamil Nadu, India, Email: {\tt n.mohanamaths@gmail.com}}\\
{\large VERNOLD VIVIN J}\footnote{Department of Mathematics, University College of Engineering Nagercoil, (A Constituent College of Anna University Chennai), Konam, Nagercoil 629 004, Tamil Nadu, India, Email: {\tt vernoldvivin@yahoo.in}}
\end{center}
\vspace{3ex}

\begin{abstract}
In this research, the exact results on $r$-dynamic coloring of lexicographic product of path with star graph, path with double star graph, path with triple star graph and finally complete graph with path are obtained.
\end{abstract}



\section{Introduction}
 The $r$-dynamic coloring was introduced by Bruce Montgomery \cite{mon}. It is a proper vertex $k$-coloring which is extended from dynamic coloring. An $r$-dynamic coloring is defined by the following mapping $c : V(G)\rightarrow \{1,2,\cdots,k\} $ it satisfy two condition:
\begin{equation*}\label{1}
c(u)\neq c(v) \hspace{6.5cm}(1.1)
\end{equation*}
\begin{equation*}\label{2}
\left|c(N(G)\right|\geqslant min\left\{r,d(v)\right\} \hspace{4cm}(1.2)
\end{equation*}
 Where, $N(v)$ denote the neighborhood vertices of $v$ and $d(v)$ denote the degree of the vertex $v$ and $k$ is the minimum positive integer of $r$-dynamic chromatic number denoted as $\chi_r (G)$. At $r=1$ the $1$-dynamic chromatic number is equal to the chromatic number of the graph and at $r=2$ it is called as dynamic chromatic number. From the literature \cite{fur},\cite{jah},\cite{tah} one can learn the wide area of $r$-dynamic chromatic number of different types of graphs and also the product of graphs are explained in \cite{ciz},\cite{gel},\cite{ham},\cite{Nes}. To the knowledge the $r$- dynamic chromatic number of lexicographic product is an starting one to enrich the topic. Moreover, we focus on the products of path with finite graphs.
\par The next section is organized with basic preliminaries and some lemmas which are useful for the main results. Then, section 3, deals with the $r$-dynamic chromatic number of lexicographic product of path with different graphs to find the exact values.
\section{Preliminaries}
In this section, we deals with basic definitions and some preliminary lemmas which are carried for the main results.
A graph $G=(V(G), E(G))$ where $V(G)$ is the set of vertices and $E(G)$ are the edges. A graph which has an identical ends at the edge are called as \textit{loop}. A graph which has no loops, and undirected graph are said to be \textit{simple}. Then a graph is said to be \textit{finite} if the order and size of the graph are finite.\par A \textit{path} between two distinct vertices $v$ and $w$ of $G$ are in sequence of ordered adjacent and $(q_0=v,q_1,q_2,\cdots,q_{l-1}=w)$ in $V(G)$ are pairwise adjacent. A graph is \textit{complete} if all the vertices are pairwise adjacent. \textit{Star $K_{1,m}$} is an tree with a vertices $\{s_0,s_1,s_2,\cdots,s_{m-1}\}$ where $s_0$ is the center vertex that are adjacent to pendant vertices $s_i$, where $1\leq i \leq m-1$. It has $m$ vertices and $m-1$ edges. \textit{Double star $K_{1,m,m}$} is obtained from star by an inclusion of $m$ pendant vertices to the existing pendant vertices $s_i$ through an edge, where $1\leq i\leq m-1$ and it has $2m+1$ vertices and $2m$ edges. \textit{Triple star $K_{1,m,m,m}$} is obtained from double star by including new $m$ pendant edge to the existing pendant vertices. It has $3m+1$ vertices and $3m$ edges.  In this work, we considered the simple, connected and undirected graph. The maximum and minimum degrees of the graph $G$ are denoted as $\delta(G)$ and $\Delta(G)$. 
\par Graph products was defined by Sabidussi and Vizing \cite{viz} and then many works has done on products but still there are many open question. In the upcoming years, graph products is an well-developing work in graph theory in that, $r$-dynamic coloring of products are described as well. Product graphs has many applications in design of interconnection networks. Also product of graph creates a larger graphs from smaller one. The  most commonly used product graphs are cartesian product, strong product, direct product and lexicographic product. The above mentioned products are commutative except lexicographic product which is associative but not commutative. In this work, we considered the lexicographic product, which was introduced by Harray (1959) but it essentially back to Hausdroff.
\par The lexicographic product {\it (graph substitution)} of graphs $G_1 [G_2]$ has vertex set $V(G_1)\times V(G_2)$ and $(v_1,w_1)(v_2,w_2)\in E(G_1 [G_2])$ whenever $(v_1,w_1)\in E(G_1)$ and $(v_2,w_2)\in E(G_2)$. The size of lexicographic product of any graphs can be generalized in the form $w_{2}v_1 +w_{1}v_2 ^2$, where $v_1=|V(G_1)|$, $v_2=|V(G_2)|$, $w_1=|E(G_1)|$ and $w_2=|E(G_2)|$. 
\begin{lema} Let $G$ be an finite, connected graph. Then, the following conditions holds:
\begin{itemize}
\item[$\ast$] $min \{r, \Delta(G)\}+ 1 \leq \chi_r(G)\leq \chi_{r+1}(G)$
%\item[$\ast$] $\chi_{r}(G)\geq min \{r, \Delta(G)\}+ 1$ 
\item[$\ast$] $\chi(G)=\chi_1(G)\leq \chi_2(G)\leq \cdots\leq \chi_{\Delta(G)}(G).$
\end{itemize}
\end{lema} 
\begin{thm} For any  positive integers $p~,t,~r$, the following results holds:\\
i) $\chi_r [K_t)] =t.$ \\
ii) For  $r\geq 2$, we get
\begin{equation*}
\chi_r [C_p)] = \left \{
\begin{array}{ll}
5, &  \mbox{for}~ p = 5\\
3, & \mbox{for} ~p = 0 (\bmod \ 3) \hspace{4cm} (1.3)\\
4, & otherwise 
 \end{array} 
\right.
\end{equation*}
\end{thm}
\begin{lema}
\cite{gel}If $G_1$ has atleast one edge, then for any graph $G_2$,\\ 
  $\chi(G_1[G_2])\geq \chi(G_1)+2\chi(G_2)-2$
\end{lema}
\begin{lema}
\cite{kla} For any graph G, if $\chi(G_2) = n$, then $\chi(G_1[G_2]) = \chi(G_1[K_n])$.
\end{lema}
\begin{lema}
\cite{kla} For any simple and finite graphs $G_1$ and $G_2$. The\\
$\chi(G_1[G_2])\leq \chi(G_1)\chi(G_2)$.
\end{lema}
In the next section, the $r$-dynamic coloring deals with lexicographic product of path with different star graphs and complete graph with path.
\section{Main Results}
\begin{thm}
 For any three positive integers $r$, $l \geq 2$ and $m\geq 3$, then \\
 $\chi_r(P_l[K_{1,m}]) = 
 \begin{cases}
4,& \mbox{for} ~ 1 \leq r \leq 3\\
r+i,& \mbox{for} ~ 4 \leq r \leq \delta-1\\
2m,& \mbox{for} ~ \delta \leq r \leq \Delta,~ l=2\\
2m,& \mbox{for} ~ r= \delta, l \neq 2 \\
2m+i,& \mbox{for} ~ \delta+1 \leq r \leq \Delta-m+1, i\epsilon N, l\neq 2\\
3m,& \mbox{for} ~ \Delta-m+2 \leq r \leq \Delta, l \leq 2
\end{cases}$ 
\end{thm}
\begin{proof} 
Let us suppose that
\begin{eqnarray*}
V[P_l] &=& \{q_0, q_1,\cdots, q_j: j\epsilon [0,l-1]\}\\
V[K_{1,m}] &=& \{s_0, s_1,\cdots, s_k: k\epsilon [0,m-1]\}
\end{eqnarray*}
  The vertices of $P_l[K_{1,m}]=\{q_{j} s_k\}$ which is of the form $j\times k$. The order of $|V(P_l[K_{1,m}])|=(l-1) \times (m-1)$. The minimum and maximum degree are $\delta(P_l[K_{1,m}])= n+1$ but when $l=2$, $\Delta(P_l[K_{1,m}])= 2n-1$ and when $l\geq 3$ we have $\Delta(P_l[K_{1,m}])= 3n-1$.  To get the faultless value of $r$-dynamic coloring of lexicographic product of path with star graph, consider the following cases: 
\begin{description}
\item{\bf Case : 1} $1\leq r \leq 3$ \\
Define a map $c_1: V(P_l[K_{1,m}])\rightarrow \{1,2,\cdots, k\}$. The $r$-dynamic coloring are as follows,
\begin{eqnarray*}
c_1 (q_{j} s_k)&=&\{1,2\} ~for~ 0\leq j\leq l-1,~ j=0,~ j~ is~ even,~ 0\leq k\leq 1\\
c_1 (q_{j} s_k)&=& \{3,4\} ~for ~0\leq j\leq l-1,~j~ is~ odd,~ 0\leq k\leq 1
\end{eqnarray*}
Then, the leftover vertices are colored as,
\begin{eqnarray*}
c_1 (q_{j} s_k)&=&2 ~for~ 0\leq j\leq l-1,~ j=0,~ j~ is~ even,~ 2\leq k\leq m-1\\
c_1 (q_{j} s_k)&=& 4 ~for ~0\leq j\leq l-1,~j~ is~ odd,~ 2\leq k\leq m-1
\end{eqnarray*}
 Thus, $\chi_r(P_l[K_{1,m}])\leq 4$. From the condition (1.2) and Lemma 1 we have, $\chi_r(P_l[K_{1,m}])\geq 4$. Therefore, $\chi_r(P_l[K_{1,m}])= 4$.
\item{\bf Case :2} $4 \leq r \leq \delta-1$\\ 
The $r$-dynamic $r+i$-coloring are as follows,\\
Define a map $c_2: V(P_l[K_{1,m}])\rightarrow \{1,2,\cdots, k\}$.\\
As $m=6$, $r=4$ and $0\leq k\leq m-1$
\begin{eqnarray*}
c_2 (q_{j} s_k)&=&\{1,2,3\} ~for~ 0\leq j\leq l-1,~ j=0,~ j~ is~ even\\
c_2 (q_{j} s_k)&=& \{4,5,6\} ~for ~0\leq j\leq l-1,~j~ is~ odd
\end{eqnarray*}
As $m=6$, $r=5$ and $0\leq k\leq m-1$
\begin{eqnarray*}
c_2 (q_{j} s_k)&=&\{1,2,3,4\} ~for~ 0\leq j\leq l-1,~ j=0,~ j~ is~ even\\
c_2 (q_{j} s_k)&=& \{5,6,7,8\} ~for ~0\leq j\leq l-1,~j~ is~ odd
\end{eqnarray*}
As $m=7$, $r=6$ and $0\leq k\leq m-1$
\begin{eqnarray*}
c_2 (q_{j} s_k)&=&\{1,2,3,4,5\} ~for~ 0\leq j\leq l-1,~ j=0,~ j~ is~ even\\
c_2(q_{j} s_k)&=& \{6,7,\cdots,r+4\} ~for ~0\leq j\leq l-1,~j~ is~ odd
\end{eqnarray*}
As $m=7$, $r=7$ and $0\leq k\leq m-1$
\begin{eqnarray*}
c_2 (q_{j} s_k)&=&\{1,2,\cdots, r-1\} ~for~ 0\leq j\leq l-1,~ j=0,~ j~ is~ even\\
c_2 (q_{j} s_k)&=& \{r,r+1,\cdots, r+5\} ~for ~0\leq j\leq l-1,~j~ is~ odd
\end{eqnarray*}
Hence, $\chi_r(P_l[K_{1,m}])\leq r+i$. Thence, from the Lemma 1 and 2 it is verified, $\chi_r(P_l[K_{1,m}])\geq r+i, ~ i=2,3,\cdots$. Therefore, $\chi_r(P_l[K_{1,m}])= r+i$.
\item{\bf Case :3} $\delta \leq r \leq \Delta,~ l=2$\\
The $r$-dynamic $2m$-coloring are as follows,\\
Define a map $c_3: V(P_l[K_{1,m}])\rightarrow \{1,2,\cdots, k\}$.\\
$c_3 (q_{j} s_k)=
\begin{cases}
1,2,\cdots, r-1 ~for~ j~=~0,~ 0\leq k\leq m-1\\
r,r+1,\cdots,2m ~ for~ j~=~1,~ 0\leq k\leq m-1
\end{cases}$\\
Therefore, $\chi_r(P_l[K_{1,m}])\leq 2m$. Hence, from the Lemma 1 and Lemma 2 it is easy check that $\chi_r(P_l[K_{1,m}])= 2m$.
\item{\bf Case :4} $r= \delta, l \neq 2$\\
The $r$-dynamic $2m$-coloring are as follows,\\
For $0\leq k\leq m-1$, define a map $c_4: V(P_l[K_{1,m}])\rightarrow \{1,2,\cdots, k\}$.\\
$c_4 (q_{j} s_k)=
\begin{cases}
1,2,\cdots, r-1 ~for~ 0\leq j\leq l-1,~ j=0,~ j~ is~ even, \\
r,r+1,\cdots,2m ~ for~0\leq j\leq l-1,~j~ is~ odd,
\end{cases}$\\
Hence, $\chi_r(P_l[K_{1,m}])\leq 2m$. Thus it is easily verified from the Lemma 1 and Lemma 2 that, $\chi_r(P_l[K_{1,m}])= 2m$.
\item{\bf Case :5} $\delta+1 \leq r \leq \Delta-m+1, l \neq 2$\\
The $r$-dynamic $2m+i$-coloring are as follows,\\
For $0\leq j\leq l-1$, $0\leq k\leq m-1$, define a map $c_5: V(P_l[K_{1,m}])\rightarrow \{1,2,\cdots, k\}$.\\
As $m=5$, $r=7$\\
$c_5:V(P_l[K_{1,m}])\rightarrow \{1,2,\cdots, 2m+1\}$\\
As $m=5$, $r=8$\\
$c_5:V(P_l[K_{1,m}])\rightarrow \{1,2,\cdots, 2m+2\}$\\
As $m=6$, $r=8$\\
$c_5:V(P_l[K_{1,m}])\rightarrow \{1,2,\cdots, 2m+1\}$\\
As $m=6$, $r=9$\\
$c_5:V(P_l[K_{1,m}])\rightarrow \{1,2,\cdots, 2m+2\}$\\
As $m=6$, $r=10$\\
$c_5:V(P_l[K_{1,m}])\rightarrow \{1,2,\cdots, 2m+3\}$\\
Thus, $\chi_r(P_l[K_{1,m}])\leq 2m+i$. Hence, from the Lemma 1 and Lemma 2 it is clearly proved as $\chi_r(P_l[K_{1,m}])= 2m+i, ~i\epsilon N$.
\item{\bf Case :6} $\Delta-m+2 \leq r \leq \Delta, l \leq 2$\\
The $r$-dynamic $3m$-coloring are as follows,\\
For $0\leq k\leq m-1,~ 0\leq j\leq l-1$, define a map $c_6: V(P_l[K_{1,m}])\rightarrow \{1,2,\cdots, k\}$.\\
$c_6 (q_{j} s_k)=
\begin{cases}
1,2,\cdots, m & ~for~  j\equiv 0(\bmod~3)\\
m,m+1,\cdots,2m & ~ for~j\equiv 1(\bmod~3)\\
2m+1,2m+2,\cdots,3m &~ for~j\equiv 2(\bmod~3)\\
\end{cases}$
\end{description} 
Hence, $\chi_r(P_l[K_{1,m}])\leq 3m$. Therefore, from the Lemma 1 and Lemma 2, it is conclude as $\chi_r(P_l[K_{1,m}])= 3m$.
\end{proof}

\begin{remark}
When $l\geq 2$ and $m=2$, the $r$-dynamic chromatic number $\chi_r(P_l[K_{1,m}])= 4$ for $1\leq r \leq 3$, since in each row we have only two vertices. Then, $\chi_{r=4}(P_l[K_{1,m}])= 5$ and $\chi_{r=5}(P_l[K_{1,m}])= 6$. The minimum and maximum degree are $\delta=3$ and $\Delta=5$.
\end{remark}

\begin{thm}
 For any three positive integers $r$, $l \geq 2$ and $m\geq 3$, we have \\
 $\chi_r(P_l[K_{1,m,m}]) = 
 \begin{cases}
4,& \mbox{for} ~ 1 \leq r \leq 3\\
r+i,& \mbox{for} ~ 4 \leq r \leq \delta-1\\
2(2m+1),& \mbox{for} ~ \delta \leq r \leq \Delta,~ l=2\\
2(2m+1),& \mbox{for} ~ r= \delta, l \neq 2 \\
2(2m+1)+i,& \mbox{for} ~ \delta+1 \leq r \leq \Delta-m+1, i\epsilon N, l\neq 2\\
3(2m+1),& \mbox{for} ~ \Delta-m+2 \leq r \leq \Delta, l \leq 2
\end{cases}$ 
\end{thm}
\begin{proof} Let us consider that
\begin{eqnarray*}
V[P_l] &=& \{q_0, q_1,\cdots, q_j: j\epsilon [0,l-1]\}\\
V[K_{1,m,m}] &=& \{s_0, s_1,\cdots, s_k: k\epsilon [0,2m]\}
\end{eqnarray*}
  The vertices of $P_l[K_{1,m,m}]=\{q_{j} s_k\}$ which is of the form $j\times k$. The order of $|V(P_l[K_{1,m,m}])|=j \times k$. The minimum and maximum degree are $\delta(P_l[K_{1,m,m}])= 2m+2$ but when $l=2$, $\Delta(P_l[K_{1,m,m}])= 3m+1$ and when $l\geq 3$ we have $\Delta(P_l[K_{1,m,m}])= 5m+2$. To get the exact value of $r$-dynamic coloring of lexicographic product of path with double star graph, consider the following cases: 
\begin{description}
\item{\bf Case : 1} $1\leq r \leq 3$ \\
The $r$-dynamic $4$-coloring are as follows,\\
Define a map $c_1: V(P_l[K_{1,m,m}])\rightarrow \{1,2,\cdots, k\}$. 
\begin{eqnarray*}
c_1 (q_{j} s_k)&=&\{1,2\} ~for~ 0\leq j\leq l-1,~ j=0,~ j~ is~ even,~ 0\leq k\leq 2m\\
c_1 (q_{j} s_k)&=& \{3,4\} ~for ~0\leq j\leq l-1,~j~ is~ odd,~ 0\leq k\leq 2m
\end{eqnarray*}
 Thus, $\chi_r(P_l[K_{1,m,m}])\leq 4$. Then, from the condition (1.2) and Lemma 1, $\chi_r(P_l[K_{1,m,m}])= 4$.
\item{\bf Case :2} $4 \leq r \leq \delta-1$\\ 
The $r$-dynamic $r+i$-coloring are as follows,\\
Define a map $c_2: V(P_l[K_{1,m,m}])\rightarrow \{1,2,\cdots, k\}$.\\
As $m=4$, $r=4$ and $0\leq k\leq 2m$
\begin{eqnarray*}
c_2 (q_{j} s_k)&=&\{1,2,3\} ~for~ 0\leq j\leq l-1,~ j=0,~ j~ is~ even\\
c_2 (q_{j} s_k)&=& \{4,5,6\} ~for ~0\leq j\leq l-1,~j~ is~ odd
\end{eqnarray*}
As $m=4$, $r=5$ and $0\leq k\leq 2m$
\begin{eqnarray*}
c_2 (q_{j} s_k)&=&\{1,2,3,4\} ~for~ 0\leq j\leq l-1,~ j=0,~ j~ is~ even\\
c_2 (q_{j} s_k)&=& \{5,6,7,8\} ~for ~0\leq j\leq l-1,~j~ is~ odd
\end{eqnarray*}
As $m=5$, $r=6$ and $0\leq k\leq 2m$
\begin{eqnarray*}
c_2 (q_{j} s_k)&=&\{1,2,3,4,5\} ~for~ 0\leq j\leq l-1,~ j=0,~ j~ is~ even\\
c_2(q_{j} s_k)&=& \{6,7,\cdots,r+4\} ~for ~0\leq j\leq l-1,~j~ is~ odd
\end{eqnarray*}
As $m=5$, $r=7$ and $0\leq k\leq 2m$
\begin{eqnarray*}
c_2 (q_{j} s_k)&=&\{1,2,\cdots, r-1\} ~for~ 0\leq j\leq l-1,~ j=0,~ j~ is~ even\\
c_2 (q_{j} s_k)&=& \{r,r+1,\cdots, r+5\} ~for ~0\leq j\leq l-1,~j~ is~ odd
\end{eqnarray*}
Therefore, $\chi_r(P_l[K_{1,m,m}])\leq r+i$. Thence, from the Lemma 1 and Lemma 2, $\chi_r(P_l[K_{1,m,m}])= r+i, ~ i=2,3,\cdots$.
\item{\bf Case :3} $\delta \leq r \leq \Delta,~ l=2$\\
The $r$-dynamic $2(2m+1)$-coloring are as follows,\\
Define a map $c_3: V(P_l[K_{1,m,m}])\rightarrow \{1,2,\cdots, k\}$.\\
$c_3 (q_{j} s_k)=
\begin{cases}
1,2,\cdots, r-1 ~for~ j~=~0,~ 0\leq k\leq 2m\\
r,r+1,\cdots,2(2m+1) ~ for~ j~=~1,~ 0\leq k\leq 2m
\end{cases}$
Hence, $\chi_r(P_l[K_{1,m,m}])\leq 2(2m+1)$. Therefore, it is clear from Lemma 1 and Lemma 2 that $\chi_r(P_l[K_{1,m,m}])= 2(2m+1)$.
\item{\bf Case :4} $r= \delta, l \neq 2$\\
The $r$-dynamic $2(2m+1)$-coloring are as follows,\\
Define a map $c_4: V(P_l[K_{1,m,m}])\rightarrow \{1,2,\cdots, k\}$.\\
For $0\leq k\leq 2m$,\\
$c_4 (q_{j} s_k)=
\begin{cases}
1,2,\cdots, r-1 ~for~ 0\leq j\leq l-1,~ j=0,~ j~ is~ even \\
r,r+1,\cdots,2(2m+1) ~ for~0\leq j\leq l-1,~j~ is~ odd
\end{cases}$
Thus $\chi_r(P_l[K_{1,m,m}])\leq 2(2m+1)$. Therefore, from Lemma 1 and Lemma 2 it is easily verified that, $\chi_r(P_l[K_{1,m,m}])= 2(2m+1)$.
\item{\bf Case :5} $\delta+1 \leq r \leq \Delta-m+1, l \neq 2$\\
The $r$-dynamic $2(2m+1)+i$-coloring are as follows,\\
Define a map $c_5: V(P_l[K_{1,m,m}])\rightarrow \{1,2,\cdots, k\}$.\\
For $0\leq j\leq l-1$, $0\leq k\leq 2m$,\\ 
As $m=5$, $r=13$\\
$c_5:V(P_l[K_{1,m,m}])\rightarrow \{1,2,\cdots, 2(2m+1)+1\}$\\
As $m=5$, $r=14$\\
$c_5:V(P_l[K_{1,m,m}])\rightarrow \{1,2,\cdots, 2(2m+1)+2\}$\\
As $m=5$, $r=15$\\
$c_5:V(P_l[K_{1,m,m}])\rightarrow \{1,2,\cdots, 2(2m+1)+3\}$\\
As $m=6$, $r=15$\\
$c_5:V(P_l[K_{1,m,m}])\rightarrow \{1,2,\cdots, 2(2m+1)+1\}$\\
As $m=6$, $r=16$\\
$c_5:V(P_l[K_{1,m,m}])\rightarrow \{1,2,\cdots, 2(2m+1)+2\}$\\
As $m=6$, $r=17$\\
$c_5:V(P_l[K_{1,m,m}])\rightarrow \{1,2,\cdots, 2(2m+1)+3\}$\\
Thus, $\chi_r(P_l[K_{1,m,m}])\leq 2(2m+1)+i$. Hence, from Lemma 1 and Lemma 2 it is clearly proved, $\chi_r(P_l[K_{1,m,m}])= 2(2m+1)+i, ~i\epsilon N$.
\item{\bf Case :6} $\Delta-m+2 \leq r \leq \Delta, l \leq 2$\\
The $r$-dynamic $3(2m+1)$-coloring are as follows,\\
Define a map $c_6: V(P_l[K_{1,m,m}])\rightarrow \{1,2,\cdots, k\}$.\\
For $0\leq k\leq 2m,~ 0\leq j\leq l-1$,\\
$c_6 (q_{j} s_k)=
\begin{cases}
1,2,\cdots, 2m, & \mbox{for}~  j\equiv 0(\bmod~3)\\
2m+1,2m+2,\cdots,2(2m+1), & \mbox{for}~j\equiv 1(\bmod~3)\\
2(2m+1)+1,\cdots,3(2m+1),& \mbox{for}~j\equiv 2(\bmod~3)
\end{cases}$
\end{description} 
Hence, $\chi_r(P_l[K_{1,m,m}])\leq 3(2m+1)$. Therefore, from Lemma 1 and Lemma 2 it is easily proved $\chi_r(P_l[K_{1,m,m}])= 3(2m+1)$.
\end{proof}

\begin{thm}
 For any three positive integers $r$, $l \geq 2$ and $m\geq 3$, we have \\
 $\chi_r(P_l[K_{1,m,m,m}]) = 
 \begin{cases}
4,& \mbox{for} ~ 1 \leq r \leq 3\\
r+i,& \mbox{for} ~ 4 \leq r \leq \delta-1\\
2(3m+1),& \mbox{for} ~ \delta \leq r \leq \Delta,~ l=2\\
2(3m+1),& \mbox{for} ~ r= \delta, l \neq 2 \\
2(3m+1)+i,& \mbox{for} ~ \delta+1 \leq r \leq \Delta-m+1, i\epsilon N, l\neq 2\\
3(3m+1),& \mbox{for} ~ \Delta-m+2 \leq r \leq \Delta, l \leq 2
\end{cases}$ 
\end{thm}
\begin{proof} The minimum degree is $\delta=3m+2$ for $l \geq 2$ and $m\geq 3$. When $l=2$, the maximum degree is $\Delta= 4m+1$ and when $l\geq 3$, the maximum degree is $\Delta= 7m+2$.   The proof of the lexicographic product of path with triple star graph can be proved in similar way from the Theorem 3. \end{proof}

\begin{lema}
 For any three positive integers $r$, $m,~n \geq 3$, we have \\
 $\chi_r(K_m[P_n]) \geq
 \begin{cases}
2m,& \mbox{for} ~ 1 \leq r \leq 2m-1\\
r+i,& \mbox{for} ~ 2m \leq r \leq \delta, ~i~=2,3,\cdots~, \mbox{i repeats (m-1)times}\\
mn,& \mbox{for} ~ r \geq \Delta
\end{cases}$ 
\end{lema}
\begin{proof} Consider that
\begin{eqnarray*}
V[K_m] &=& \{p_0, p_1,\cdots, p_j: j\epsilon [0,m-1]\}\\
V[P_n] &=& \{q_0, q_1,\cdots, q_k: k\epsilon [0,n-1]\}
\end{eqnarray*}
  The vertices of $K_m[P_n]=\{p_{j} q_k\}$ which is of the form $j\times k$. The order of $|V(K_m[P_n])|=j \times k$. The minimum and maximum degree are $\delta(K_m[P_n])= (m-1)n+1$ and $\Delta(K_m[P_n])= (m-1)n+2$. \\\\
   For $1 \leq r \leq 2m-1$, select $r=1$, from Lemma 1 $\chi_r(G)\geq min \{r,\Delta(G)\}+1$, $\chi_r(K_m[P_l])\geq min\{1,\Delta(K_m[P_l])\}+1 \geq 2$, so $\chi_r(K_m[P_l])\geq 2$. Next, select $r=2$, from Lemma 1, $\chi_r(K_m[P_l])\geq 3$. In similar way select $r=2m-1$, $\chi_r(K_m[P_l])\geq min\{2m-1,\Delta(K_m[P_l])\}+1=2m$. Thus, $\chi_r(K_m[P_l])\geq 2m$, for $1 \leq r \leq 2m-1$.\\\\
For $2m \leq r \leq \delta$, choose $r=2m$, $\chi_r(K_m[P_n])\geq min\{2m,\Delta(K_m[P_n])\}+1 = 2m+1$. Thus, $\chi_r(K_m[P_n]\geq r+i$, for $2m \leq r \leq \delta$.\\\\
Similarly, for $r \geq \Delta$ choose $r=(m-1)n+2$, $\chi_r(K_m[P_n]\geq min\{(m-1)n+2,\Delta(K_m[P_n])\}+1 \geq (m-1)n+2$. Therefore, $\chi_r(K_m[P_n]\geq mn$, for $r \geq \Delta$. So, we must have $mn$ colors. \end{proof}

\begin{thm}
 For any three positive integers $r$, $m,~n \geq 3$, we have \\
 $\chi_r(K_m[P_n]) =
 \begin{cases}
2m,& \mbox{for} ~ 1 \leq r \leq 2m-1\\
r+i,& \mbox{for} ~ 2m \leq r \leq \delta, ~i~=2,3,\cdots~,\mbox{i repeats (m-1)times}\\
mn,& \mbox{for} ~ r \geq \Delta
\end{cases}$ 
\end{thm}
\begin{proof} Considering the $r$-dynamic coloring as follows: 
\begin{description}
\item{\bf Case : 1} $1\leq r \leq 2m-1$ \\
Define a map $c_1: V(K_m[P_n])\rightarrow \{1,2,\cdots, k\}$. The $r$-coloring are as follows,\\
For $0\leq j\leq m-1$,
\begin{eqnarray*}
c_1 (p_{j} q_k)&=&\{1,3,\cdots,2m-1\}, ~for~ 0\leq k\leq n-1,~ k=0,~ k~ is~ even\\
c_1 (q_{j} s_k)&=& \{2,4,\cdots, 2m\}, ~for ~ 0\leq k\leq n-1,~j~ is~ odd
\end{eqnarray*}
 Hence, $\chi_r(K_m[P_n])\leq 2m$. Thus  condition (1.2) and Lemma 5 holds. Therefore, $\chi_r(K_m[P_n])= 2m$.
\item{\bf Case :2} $2m \leq r \leq \delta$\\
Define a map $c_2: V(K_m[P_n])\rightarrow \{1,2,\cdots, k\}$.\\
For $0\leq j\leq m-1$, $0\leq k\leq n-1$, the $r$-dynamic coloring are as follows,\\ 
As $m=4$, $n=5$\\
$\chi_{r=2m} :c_2(V(K_m[P_n]))\rightarrow \{1,2,\cdots, 2m+2\}$\\
As $m=4$, $n=5$\\
$\chi_{r=2m+1} :c_2(V(K_m[P_n]))\rightarrow \{1,2,\cdots, 2m+3\}$\\
As $m=4$, $n=6$\\
$\chi_{r=2m+2} :c_2(V(K_m[P_n]))\rightarrow \{1,2,\cdots, 2m+4\}$\\
Similarly,\\
As $m=5$, $n=6$\\
$\chi_{r=2m} :c_2(V(K_m[P_n]))\rightarrow \{1,2,\cdots, 2m+2\}$\\
As $m=5$, $n=6$\\
$\chi_{r=2m+1} :c_2(V(K_m[P_n]))\rightarrow \{1,2,\cdots, 2m+3\}$\\
$\cdots \cdots$\\
As $m=5$, $n=6$\\
$\chi_{r=2m+3} :c_2(V(K_m[P_n]))\rightarrow \{1,2,\cdots, 2m+5\}$\\
As $m=5$, $n=6$\\
$\chi_{r=2m+4} :c_2(V(K_m[P_n]))\rightarrow \{1,2,\cdots, 2m+7\}$\\
$\cdots \cdots$\\
As $m=5$, $n=6$\\
$\chi_{r=2m+7} :c_2(V(K_m[P_n]))\rightarrow \{1,2,\cdots, 2m+10\}$\\
Thus, the $i$-value repeats for $(m-1)$ times and $\chi_r(K_m[P_n])\leq r+i$. Hence, the Lemma 5 holds, $\chi_r(K_m[P_n])= r+i$, for $i=2,3,\cdots$ and $i$-value repeats for $(m-1)$ times.
\item{\bf Case :3} $r \geq \Delta$\\
The $r$-dynamic $mn$-coloring are as follows;\\
Define a map $c_3: V(K_m[P_n])\rightarrow \{1,2,\cdots, k\}$.
\begin{eqnarray*}
c_3 (p_{j} q_k)&=&\{1,2,3,\cdots,mn\} ~for~ 0\leq j\leq m-1,~ 0\leq k\leq n-1
\end{eqnarray*}
\end{description}
Thus, $\chi_r(K_m[P_n])\leq mn$. Based on the Lemma 5, it is an easy clarification that $\chi_r(K_m[P_n])= mn$.
\end{proof}

\begin{remark}
For three positive integers, $r$, $n=2$ and $m\geq 3$, the $r$-dynamic chromatic number of $\chi_r(K_m[P_2])= 2m$. The degrees are $\delta=\Delta=2m-1$. Since, the $r$-dynamic coloring of complete graph with $n$ vertices is with $n$-colors. Thus, the lexicographic product of complete graph with path is also forms an complete graph. Therefore, $\chi_r(K_m[P_2])= 2m$. And also $\chi_r(K_m[P_2])= \chi_r(P_2[K_n])$ 
\end{remark}

\end{document}